\documentclass[a4paper, 10pt]{amsart}

\usepackage{amsmath, amsfonts, amssymb}
\usepackage{amsthm}
\usepackage{esint}
\usepackage[%colorlinks,
bookmarks=false]{hyperref}
\usepackage{cite}
\usepackage{array}
\usepackage{amsaddr}
\usepackage{enumerate}
\title{\large{Uncertainty inequalities on groups and homogeneous spaces via isoperimetric inequalities}}

\author{Gian Maria Dall'Ara\and Dario Trevisan}
\address{Scuola Normale Superiore, Pisa, Italy}
\email{gianmaria.dallara@sns.it, dario.trevisan@sns.it}

\date{\today}

\newcommand{\R}{\mathbb{R}}
\newcommand{\N}{\mathbb{N}}

\newcommand{\C}{\mathbb{C}}

\newcommand{\be}{\begin{equation*}}
\newcommand{\ee}{\end{equation*}}
\newcommand{\bel}{\begin{equation}}
\newcommand{\eel}{\end{equation}}
\newcommand{\bee}{\begin{eqnarray*}}
\newcommand{\eee}{\end{eqnarray*}}
\newcommand{\eps}{\varepsilon}

\newcommand{\comma}{\textrm{,}}
\newcommand{\period}{\textrm{.}}
\newcommand{\bra}[1]{\left( #1 \right)}

\newcommand{\cur}[1]{\left\{ #1 \right\}}

\newcommand{\abs}[1]{\left| #1 \right|}
\newcommand{\nor}[1]{\left\| #1 \right\|}
\newcommand{\cnum}{\mathbb{C}}

\newtheorem{thm}{Theorem}
\newtheorem{lem}[thm]{Lemma}

\newtheorem{cor}[thm]{Corollary}

\begin{document}

\maketitle

\begin{abstract}
We prove a new family of $L^p$ uncertainty inequalities on fairly general groups and homogeneous spaces, both in the smooth and in the discrete setting. The novelty of our technique consists in the observation that the $L^1$ endpoint can be proved by means of appropriate isoperimetric inequalities.
\end{abstract}

\section{Introduction}

The aim of this article is to prove a family of uncertainty inequalities on fairly general groups and homogeneous spaces, both in the smooth and in the discrete setting, highlighting a connection with isoperimetric inequalities. The basic example of the kind of uncertainty inequalities we are interested in is the classical inequality \bel\label{heisen}
\nor{f}_{L^2(\R^N)}^2\leq 4N^{-2}\nor{\nabla f}_{L^2(\R^N)} \nor{\abs{x}f}_{L^2(\R^N)}\qquad\forall f\in C^\infty_c(\R^N),
\eel which first appeared (for $N=1$) in Appendix $1$ of Hermann Weyl's celebrated book ``The Theory of Groups and Quantum Mechanics''. We refer to the beautiful survey \cite{fo-si} for the higher-dimensional versions and for many other results related to the uncertainty principle.

In this introductory section, we give an informal description of our results, postponing more precise and slightly more general statements to Theorem \ref{HPW} in Section \ref{smooth-setting} (for the smooth setting) and to Theorem \ref{HPW-discrete} in Section \ref{discrete-setting} (for the discrete setting).
We work on $M$, a homogenous space for a (Lie or finitely generated) group $G$ such that the isotropy subgroups are compact, and endowed with:\begin{enumerate}[(i)]
\item an invariant measure $\mu$, with respect to which $L^p$ spaces are defined,
\item an invariant distance $d$,
\item an invariant gradient $\nabla$, naturally coupled with $d$.
\end{enumerate} In this setting, one can define the crucial growth function $\Gamma_M(r):=\mu(B(r))$, where $B(r)$ is any ball of radius $r$ with respect to the distance $d$. Our main result is that, for a wide class of spaces $M$ as above, the inequality\bel\label{model}
\nor{f}_{L^p(M)}^2\leq Cp\nor{\nabla f}_{L^p(M)} \nor{wf}_{L^p(M)}
\eel holds for sufficiently nice functions $f:M\rightarrow \R$ and any non-negative $w$ satisfying the growth condition \bel\label{growth-intro}
\mu(w\leq r)\leq \Gamma_M(r)\qquad \forall r.
\eel The constant $C$ is independent of $M$ and explicitly computable (by a careful reading of our arguments), albeit non optimal (see Section \ref{question} for more on this point). Notice that the distance from a fixed point, $d(\cdot, m_0)$, satisfies condition \eqref{growth-intro}, so that \eqref{heisen} is a particular instance of \eqref{model}, up to the optimal constant. To appreciate the greater generality provided by condition \eqref{growth-intro} see the fourth remark to Theorem \ref{HPW} below.

Let us stress that, to our knowledge, the $L^1$ estimates we consider were never studied before and they are new also when $M=\R^N$. 

\subsection{Examples} We list a few examples of groups and homogeneous spaces to which our results apply.
\begin{enumerate}[(i)]
\item Any compact or non-compact Riemannian symmetric space endowed with the invariant measure, distance and gradient.
\item Any unimodular connected Lie group, endowed with a system of left invariant vector fields $X_1,\dots,X_k$ generating its Lie algebra. The gradient is $\nabla :=(X_1,\dots,X_k)$ and the distance is the control metric associated to these vector fields (see the book \cite{va-sa-cou}).
\item The unit sphere $S^{2N-1}$ in $\C^N\equiv\R^{2N}$, endowed with the natural action of $U(N)$ and the $U(N)$-invariant sub-Riemannian structure in which the horizontal bundle is given by the complex tangent directions (see \cite{invitation}). The gradient and the distance are those naturally attached to this structure (see Section \ref{smooth-setting}).
\item Cayley graphs of any finitely generated group, endowed with the word metric and the graph gradient (as defined e.g.\ in Section $1$ of \cite{coulhon-saloffcoste}).
\item More generally, any Schreier coset graph, as described in Section \ref{assumptions-discrete}. 
\end{enumerate}

\subsection{Comments on the proof technique}
The basic observation, which motivates the title of the paper, is that the $L^1$ case of \eqref{model}, to which the $L^p$ case can be reduced, is related to a weak isoperimetric inequality on the space $M$. In fact, well-established techniques in geometric measure theory allow to define a notion of perimeter for subsets of $M$, which is naturally associated to the gradient above. The weak isoperimetric inequality then reads as follows: \emph{a set $E\subseteq M$, having the same measure as a ball of radius $r$, has necessarily perimeter greater than or equal to $C\frac{\Gamma_M(r)}{r}$}, where $C$ is an explicitly computable constant independent of $M$. We are able to show that this isoperimetric inequality implies the main estimate \eqref{model} for $p=1$. The weak isoperimetric inequality itself follows from an established circle of ideas, a brief description of which can be found e.g.\ in Section $6.43$ of \cite{gro-green}. For the sake of completeness, we dedicate the Appendix to a short proof of it along these lines. Notice that isoperimetric results close to the one above appear in \cite{coulhon-saloffcoste}, where the authors deal with many settings partially intersecting ours.

\subsection{Comparison with the existing literature} There has recently been some work in the direction of establishing in very general settings uncertainty inequalities of the kind we are interested in. 

First we mention the works of F.\ Ricci \cite{ricci-polynomial}, P.\ Ciatti, F.\ Ricci and M.\ Sundari \cite{cia-ri-sun2}, and A.\ Martini \cite{martini}, which prove very general $L^2$ uncertainty inequalities, in which the gradient may be replaced by a positive power of a non-negative operator satisfying certain assumptions. Their approach is spectral-theoretic and relies on heat kernel techniques, which does not seem to yield our $L^2$ inequalities, at least when the volume growth function grows faster than any polynomial at infinity.

We also mention the work of A.\ Okoudjou, L.\ Saloff-Coste and A.\ Teplyaev \cite{sa-ok-tep}, where they prove in particular $L^2$ uncertainty inequalities for general  non-compact unimodular sub-Riemannian Lie groups and finitely generated groups (Section $3.5$ of \cite{sa-ok-tep}). Their approach is based on Poincar\'e-type and Nash-type inequalities, and their results are confined to $w=d(\cdot,m_0)$.

Concerning $L^p$ uncertainty inequalities for $p\neq2$, we point out the recent paper \cite{hardy-unc-strat} of P.\ Ciatti, M.\ Cowling and F.\ Ricci. Theorem C of that paper, combined with known facts about Riesz transforms, allows in particular to derive our Theorem \ref{HPW} when $M$ is a stratified group, $p>1$ and $w$ is the homogeneous norm.

%Of course, there are well-studied links between isoperimetric inequalities, heat kernel estimates and Poincar\'e-type inequalities (see, e.g.\ the books \cite{va-sa-cou} and \cite{aspects}), but the precise correspondences between these properties are very complicated and not easily exploitable in our general context.

\subsection{Plan of the paper} Sections from \ref{smooth-setting} to \ref{compact-section} are dedicated to (slight generalizations of) inequality \eqref{model} in the smooth setting, while sections \ref{discrete-setting} and \ref{proof-discrete} describe our discrete setting and illustrate the modifications needed to extend our arguments to cover the discrete case. Finally, Section \ref{question} comments on some quantitative aspects of our estimates.

We conclude with a few comments on our notation. Any $C$ appearing in an estimate stands for an absolute constant which is independent of everything, in particular of the group or homogeneous manifold one is working with. Moreover we denote by $\dot{\gamma}(t)$ the time derivative of a curve, by $f^\leftarrow(A)$ the inverse image of the set $A$ with respect to the function $f$, and by $1_A$ the characteristic function of $A$.

\section{Smooth setting}\label{smooth-setting} 

We recall that a \emph{sub-Riemannian structure} on a connected smooth manifold $M$ is a pair $(\mathcal{V},g)$, where $\mathcal{V}$ is a completely non-integrable distribution on $M$ and $g$ is a smooth fiber metric on $\mathcal{V}$. Here we adopt the differential geometric terminology and by a distribution we mean a constant rank sub-bundle of the tangent bundle of $M$. The complete non-integrability means that the tangent bundle $TM$ is generated by iterated commutators of smooth sections of $\mathcal{V}$. In the usual jargon, $\mathcal{V}$ is called the \emph{horizontal distribution} of the sub-Riemannian manifold and smooth sections of $\mathcal{V}$ are called \emph{horizontal vector fields}. Finally, for every $m\in M$, $g_m$ is a scalar product on the fiber $\mathcal{V}_m$ depending smoothly on $m$. Of course when $\mathcal{V}=TM$, $(M,g)$ is a Riemannian manifold.

Several geometric and analytic objects are naturally attached to a sub-Riemannian manifold $(M,\mathcal{V},g)$: we will be mainly dealing with the \emph{Carnot-Carath\'eodory distance} and the \emph{horizontal gradient}.\newline One can define the Carnot-Carath\'eodory distance as\be
d_{CC}(m,n):=\inf\text{ length}(\gamma)\qquad (m,n\in M),
\ee where the $\inf$ is taken as $\gamma$ varies over the piecewise $C^1$ curves $\gamma:[0,T]\rightarrow M$ that connect $m$ to $n$ and are horizontal, i.e.\ \be
\dot{\gamma}(t)\in \mathcal{V}_{\gamma(t)}\setminus\{0\}
\ee for every $t\in [0,T]$ (at the non-differentiability points, one requires that both the left and the right derivatives satisfy the property). The length is computed in terms of $g$:\be
\text{length}(\gamma):=\int_0^T\sqrt{g_{\gamma(t)}(\dot{\gamma}(t),\dot{\gamma}(t))}dt.
\ee 
It is well known that $d_{CC}$ is a distance inducing the manifold topology on $M$ (Theorem $2.1.2$ and Theorem $2.1.3$ of \cite{montgomery})  and we will denote by $B(m,r)$ the open ball of center $m$ and radius $r$ with respect to this distance. We denote $d_0$ the diameter of the metric space $(M,d_{CC})$.

If $f:M\rightarrow \R$ is regular ($C^1$ is enough), one can define the \emph{horizontal gradient} $\nabla_Hf$ as the unique horizontal vector field satisfying the identity\be
g(\nabla_Hf,X)=Xf\qquad \text{ for every horizontal vector field } X.
\ee 

Assume now that $M$ carries a smooth left action of a Lie group $G$. To fix the notation, we say that an element $x\in G$ acts on $m\in M$ sending it to $x\cdot m$ and we denote by $\phi_x$ the diffeomorphism $m\mapsto x\cdot m$. We say that the sub-Riemannian structure $(\mathcal{V},g)$ is \emph{$G$-invariant} if the differential $d(\phi_x):TM\rightarrow TM$ induces an isometry of the sub-bundle $\mathcal{V}$, i.e.\ if the restriction \be
d(\phi_x)_{|m}:(\mathcal{V}_m,g_m)\longrightarrow (\mathcal{V}_{x\cdot m},g_{x\cdot m})
\ee
is a well-defined isometry for every $x\in G$ and $m\in M$.\newline If the action of $G$ is transitive, we call $M$ a \emph{sub-Riemannian homogeneous manifold for $G$}. One may easily check that in this case the Carnot-Carath\'eodory distance is $G$-invariant, i.e.\ \be
d_{CC}(x\cdot m,x\cdot n)=d_{CC}(m,n)\qquad\forall x\in G,\ m,n\in M,
\ee and that the horizontal gradient is also $G$-invariant:\be
\nabla_H(f\circ\phi_x)=(\nabla_Hf)\circ\phi_x\qquad\forall f\in C^1(M),\ x\in G.
\ee

\subsection{Standing assumptions and statement of the main result in the smooth setting}\label{statement-smooth}

We are now in a position to state our \emph{standing assumptions} in the smooth setting:\begin{enumerate}[(i)]
\item $M$ is a connected smooth manifold,
\item $G$ is a connected and unimodular Lie group,
\item $G$ acts smoothly and transitively on $M$,
\item the isotropy subgroup of some, and hence every, point of $M$ is compact,
\item $(\mathcal{V},g)$ is a $G$-invariant sub-Riemannian structure on $M$.
\end{enumerate}
Under these assumptions there always exists a $G$-invariant Borel measure which is finite on compact sets and unique up to positive multiples (see the Appendix for more details). Choosing such a measure $\mu$ allows to define $G$-invariant $L^p$ spaces on $M$ and the crucial \emph{volume growth function}\be
\Gamma_M(r):=\mu(B(m,r))\qquad\forall r\geq 0,
\ee which is well-defined since the right hand side is independent of $m$, thanks to the invariance of both the measure and the distance $d_{CC}$, and the transitivity of the action.

We are finally in a position to state our main result in the smooth setting.

\begin{thm}\label{HPW}
Assume that $M$ satisfies the assumptions above and that $w:M\rightarrow[0,+\infty]$ is a Borel function such that\bel\label{hyphpw}
\mu\{m\in M:\ w(m)\leq r\}\leq \Gamma_M(r)\qquad\forall r\geq0.
\eel If $M$ is not compact, then the inequality\bel\label{hpw}
\nor{f}_{L^p(M)}\leq Cp^\frac{\alpha}{\alpha+1}\nor{\nabla_H f}_{L^p(M)}^\frac{\alpha}{\alpha+1}\nor{w^\alpha f}_{L^p(M)}^\frac{1}{\alpha+1}
\eel holds for every $f\in C_c^\infty(M)$ and $\alpha>0$.\newline 
If $M$ is compact, inequality \eqref{hpw} holds for every $f$ that satisfies the additional assumption\bel\label{zero-mean}
\int_Mfd\mu=0.
\eel
\end{thm}

We would like to highlight some features of our result:\begin{enumerate}[(i)]

\item The constant in \eqref{hpw} is independent of the space $M$.
\item Both the hypothesis and the conclusion are $G$-invariant (e.g.\ $w$ satisfies the hypothesis \eqref{hyphpw}  if and only if $w\circ\phi_x$ satisfies the same hypothesis, where $x\in G$) and they do not depend on the choice of the $G$-invariant measure $\mu$.
\item A very natural example of a function $w$ satisfying the assumptions of Theorem \ref{HPW} is the Carnot-Carath\'eodory distance to a fixed point, i.e.\ $d_{CC}(m_0,\cdot)$ ($m_0\in M$).
With this choice of $w$ and for $p=2$ and $\alpha=1$, inequality \eqref{hpw} is the most natural generalization of \eqref{heisen} to the sub-Riemannian homogeneous setting.
\item Fix $m_1,\dots,m_k\in M$ and consider \be
w(m):=k\min_{j=1,\dots,k}d_{CC}(m_j,m)\qquad(m\in M).
\ee From the super-additivity of $\Gamma_M$ discussed in Section \ref{tech} it follows that $w$ satisfies \eqref{hyphpw} (at least in the non-compact case, in the compact case one needs to replace $k$ with $Ck$ for some absolute constant $C$). Inequality \eqref{hpw} shows that there is a limit to the extent to which a normalized low-energy $L^2$ function on $M$ ($\nor{f}_{L^2(M)}=1$ and $\nor{\nabla_Hf}_{L^2(M)}$ small) can be localized around a finite configuration of points. Notice that this cannot be deduced from the inequality with $w=d_{CC}(m_0,\cdot)$ even in the Euclidean setting.
\item In case $M$ is compact or, equivalently, $(M,d_{CC})$ has finite diameter, the hypothesis \eqref{hyphpw} is empty for $r$ greater than or equal to the diameter. The additional restriction $\int_Mfd\mu=0$ is of course necessary, since constant functions are in $C^\infty_c(M)$. 

\end{enumerate}

\subsection{Structure of the proof}\label{structure-proof}

The proof of Theorem \ref{HPW} occupies sections from \ref{uncertainty-isoperimetric} to \ref{compact-section} of the paper. For the sake of clarity, in sections \ref{uncertainty-isoperimetric}, \ref{isoperimetric-section} and \ref{Lp} we assume that $M$ is non-compact, since this simplifies several aspects of the proof. We then devote Section \ref{compact-section} to the technical modifications needed to deal with the compact case.\newline
The proof in the non-compact case consists of three main steps.
\begin{enumerate}[(a)]

\item In Section \ref{uncertainty-isoperimetric} we start with the $L^1$ inequality and we show how it can be reduced to a basic gradient estimate (Lemma \ref{Faris}) for functions on $M$.

\item Then, in Section \ref{isoperimetric-section}, we deduce the gradient estimate from the weak isoperimetric inequality quoted in the introduction. The main tools are the coarea formula and a super-additivity property for $\Gamma_M$.

\item Finally, in Section \ref{Lp}, the deduction of the $L^p$ estimates from the $L^1$ estimate is a standard trick exploting Leibniz rule and H\"older inequality. The revision of this step in the compact case is a bit painful, due to the restriction \eqref{zero-mean}.

\end{enumerate}

\section{Reduction of the $L^1$ inequality to a gradient estimate}\label{uncertainty-isoperimetric}

As anticipated in Section \ref{structure-proof}, we assume that $M$ satisfies the assumptions of Section \ref{statement-smooth} and is non-compact. We will not remove this restriction until Section \ref{compact-section}, and in sections \ref{uncertainty-isoperimetric} and \ref{isoperimetric-section} we will only deal with the $L^1$ inequality.

For the sake of clarity, we state the $L^1$ inequality of Theorem \ref{HPW} in the non-compact setting.

\begin{thm}\label{HPWL1}
Assume that $M$ is non-compact and satisfies the assumptions of Section \ref{statement-smooth}, and that $w:M\rightarrow[0,+\infty]$ is a Borel function such that\be
\mu\{m\in M:\ w(m)\leq r\}\leq \Gamma_M(r)\qquad\forall r\geq0.
\ee Then the inequality\be
\nor{f}_{L^1(M)}\leq C\nor{\nabla_H f}_{L^1(M)}^\frac{\alpha}{\alpha+1}\nor{w^\alpha f}_{L^1(M)}^\frac{1}{\alpha+1}
\ee holds for every $f\in C_c^\infty(M)$ and $\alpha>0$.
\end{thm}

The first observation to be made is that the estimate in Theorem \ref{HPWL1} is multiplicative in nature, but can nevertheless be reduced to an additive inequality, by the elementary identity\be
\min_{r>0}\left(ar+br^{-\alpha}\right)=\left(\alpha^\frac{1}{\alpha+1}+\alpha^{-\frac{\alpha}{\alpha+1}}\right)a^\frac{\alpha}{\alpha+1}b^\frac{1}{\alpha+1},
\ee which holds for every $a,b,\alpha>0$. Notice that the constant depends mildly on $\alpha$, since \be
\alpha^\frac{1}{\alpha+1}+\alpha^{-\frac{\alpha}{\alpha+1}}=e^{\frac{\log(\alpha)}{\alpha+1}}+e^{\frac{\log(1/\alpha)}{1/\alpha+1}}\leq 2e.
\ee
Hence Theorem \ref{HPWL1} follows from the existence of a universal constant $C$ such that the additive inequality\bel\label{additiveHPW}
\nor{f}_{L^1(M)}\leq Cr\nor{\nabla_H f}_{L^1(M)}+Cr^{-\alpha}\nor{w^\alpha f}_{L^1(M)}
\eel holds for every $r>0$. Our task is then to prove estimate \eqref{additiveHPW}. Since \bel\label{trivial-estimate}
\nor{f}_{L^1(M)}=\int_{w\leq r}|f|d\mu+\int_{w> r}|f|d\mu\leq \int_{w\leq r}|f|d\mu+r^{-\alpha}\int_M w^\alpha |f|d\mu,
\eel what we will establish is the following gradient estimate.

\begin{lem}\label{Faris}
If $E\subseteq M$ is Borel and $\mu(E)\leq \Gamma_M(r)$, then the inequality\bel\label{faris-estimate}
\int_E|f|d\mu\leq Cr\int_M|\nabla_Hf|d\mu
\eel holds for every $f\in C^\infty_c(M)$. 
\end{lem}

\begin{proof}[Proof that Lemma \ref{Faris} implies Theorem \ref{HPWL1}]
By assumption, $\mu(w\leq r)\leq \Gamma_M(r)$ and hence Lemma \ref{Faris} gives\be
\int_{w\leq r}|f|d\mu\leq Cr\int_M|\nabla_Hf|d\mu.
\ee This, together with \eqref{trivial-estimate}, gives the additive inequality \eqref{additiveHPW}. By the discussion above, we are done.
\end{proof}

\section{Proof of Lemma \ref{Faris} via the weak isoperimetric inequality}\label{isoperimetric-section}

We start by recalling the suitable notion of perimeter for subsets of $M$, adapted to the sub-Riemannian structure: the \emph{horizontal perimeter} of $E\subseteq M$ is defined by
\be \nor{\partial_H E}:= \inf \cur{ \liminf_{n \to \infty} \int_M \abs{\nabla_H f_n} d\mu : \{f_n\}_{n\in \N}\subseteq C^\infty_c(M) \text{ s.t. } f_n\rightarrow 1_E \text{ in }L^1(M)}.\ee
Of course, $\nor{\partial_HE}$ is always a well-defined element of $[0,+\infty]$.
 
The horizontal perimeter satisfies the following isoperimetric property.

\begin{thm}[Weak isoperimetric inequality]\label{weak-isoperimetric}
Let $M$ be a non-compact manifold satisfying the assumptions of Section \ref{statement-smooth}. If $E\subseteq M$ is a Borel set such that $\mu(E)\leq \Gamma_M(r)$, then\be
\mu(E)\leq Cr\nor{\partial_HE}.
\ee
\end{thm} 
We sketch a proof of Theorem \ref{weak-isoperimetric} in the Appendix.

We call this a \emph{weak} isoperimetric inequality since, while it provides a quantitative lower bound for the perimeter of a set in terms of its measure, and it does it comparing the set with metric balls, it does not say which sets minimize the perimeter among those of a fixed volume.
 
Another respect in which Theorem \ref{weak-isoperimetric} is not sharp can be made apparent considering the Euclidean case, i.e.\ $M=\R^N$. The Euclidean isoperimetric inequality (see e.g. \cite{federer}, $3.2.43$) states that if $E$ has the same volume as a Euclidean ball of radius $r$, its perimeter must be greater than the perimeter of this ball, which equals $N\frac{\Gamma_{\R^N}(r)}{r}$. This is much better than Theorem \ref{weak-isoperimetric} when $N$ goes to infinity. Despite these limitations, Theorem \ref{weak-isoperimetric} has the advantage of being applicable at our level of generality.

We are now in a position to prove Lemma \ref{Faris}. We claim that it follows from a stronger version of Theorem \ref{weak-isoperimetric}, which we now state.

\begin{thm}[Weak isoperimetric inequality, stronger form]\label{weak-isoperimetric-decoupled}
If $\mu(E)\leq \Gamma_M(r)$, the following inequality holds for every Borel set $A$ of finite measure:\be
\mu(A\cap E)\leq Cr\nor{\partial_H A}.
\ee
\end{thm} Choosing $A=E$, we immediately see that Theorem \ref{weak-isoperimetric-decoupled} implies Theorem \ref{weak-isoperimetric}.\newline 
To obtain Lemma \ref{Faris} from Theorem \ref{weak-isoperimetric-decoupled} we need the coarea formula 
\bel\label{co} \int_0^\infty \nor{ \partial_H \{|f|>s\}  } ds =  \int_M|\nabla_H |f||d\mu, \eel which holds for any test function $f$ (a proof may be obtained by adapting in a straightforward way Proposition $4.2$ of \cite{mi}). Given $f\in C^\infty_c(M)$ and $E\subseteq M$ such that $\mu(E)\leq\Gamma_M(r)$, we can apply Theorem \ref{weak-isoperimetric-decoupled} to $A_s:=\{|f|>s\}$, obtaining\be
\mu(A_s\cap E)\leq Cr\nor{\partial_HA_s}\qquad\forall s>0.\ee Integrating both sides and applying the coarea formula \eqref{co}, we find\be
\int_E|f|d\mu\leq Cr\int_0^{+\infty}\nor{\partial_HA_s}ds=Cr\int_M|\nabla_H|f||d\mu\leq  Cr\int_M|\nabla_Hf|d\mu.
\ee
Actually, one can do the converse and obtain Theorem \ref{weak-isoperimetric-decoupled}  from Lemma \ref{Faris}, recalling the definition of horizontal perimeter.

\subsection{Proof of Theorem \ref{weak-isoperimetric-decoupled}}\label{tech} The key fact is a \emph{super-additivity} property of the volume growth function $\Gamma_M$:
\be \Gamma_M(r+s)\geq \Gamma_M(r)+\Gamma_M(s)\qquad\forall r,s\geq 0.\ee

We first show how Theorem \ref{weak-isoperimetric-decoupled} may be obtained by means of this fact, and then we prove it. 

We introduce, for a Borel set $S\subseteq M$, the notation\be
r_S:=\inf\{t\geq0: \mu(S)\leq \Gamma_M(t)\}.
\ee

Given $E,A\subseteq M$ with $\mu(E)\leq\Gamma_M(r)$, we assume without loss of generality that $\mu(A\cap E)>0$ and hence that $r_{A\cap E}>0$. Notice that $r_{A\cap E}\leq r_A,r$. Theorem \ref{weak-isoperimetric} applied to $A$ gives\bel\label{rA}
\mu(A)\leq Cr_A||\partial_HA||.
\eel If $r_A=r_{A\cap E}$, this inequality immediately implies Theorem \ref{weak-isoperimetric-decoupled}. We can then assume $r_{A\cap E}<r_A$ and choose $t,t'>0$ such that\be
r_{A\cap E}<t<t'<r_{A}\text{ and } \left\lfloor{\frac{t'}{t}}\right\rfloor \geq\frac{r_A}{2r_{A\cap E}},
\ee where $\lfloor\cdot\rfloor$ denotes the integer part. We have\be
\mu(A)> \Gamma_M(t')\geq \Gamma_M\left(\left\lfloor{\frac{t'}{t}}\right\rfloor t\right)\geq \left\lfloor{\frac{t'}{t}}\right\rfloor\Gamma_M(t)\geq\frac{r_A}{2r_{A\cap E}}\mu(A\cap E),
\ee where the second inequality follows from the claimed super-additivity of $\Gamma_M$. From this and \eqref{rA} we conclude, since $r_{A\cap E}\leq r$.

Notice that one may alternatively prove that $\Gamma_M$ is continuous and hence that $\Gamma_M(r_S)=\mu(S)$, slightly simplifying the above argument. On the other hand, this argument has the advantage of working also in the discrete setting, where $\Gamma_M$ is usually not continuous.
\newline

For the proof of the super-additivity, the first observation is that $(M,d_{CC})$ is a \emph{complete and locally compact path metric space} in the sense of Definition $1.7$ of \cite{gro-green}. The fact that it is a path metric space is elementary, while the Ball-Box Theorem for sub-Riemannian manifolds (Theorem $2.4.2$ of \cite{montgomery}) implies that $d_{CC}$ induces the manifold topology and hence that the metric space $(M,d_{CC})$ is locally compact. To prove the completeness, we observe that if we fix an arbitrary $0\in M$ the Ball-Box Theorem gives a $\delta>0$ such that $B(0,\delta)$ is pre-compact. By $G$-invariance, $B(m,\delta)$ is also pre-compact for any $m\in M$. If $\{x_n\}_{n\in N}$ is a Cauchy sequence, there is an $n_0$ such that $\{x_n\}_{n\geq n_0}$ is contained in $B(m,\delta)$ for some $m\in M$. Since this set is pre-compact, one can extract a convergent subsequence from $\{x_n\}_{n\geq n_0}$. This implies that the original sequence converges and hence that $(M,d_{CC})$ is complete. We can now apply the metric Hopf-Rinow Theorem (see \cite{gro-green}, p.9) to conclude that every open ball is pre-compact and that there is a minimizing geodesic connecting any pair of points of $M$. Since $\mu$ is finite on compact sets, it follows in particular that $\Gamma_M(r)<\infty$ for every $r\geq0$.

We are now in a position to prove the super-additivity of $\Gamma_M$. Take two points whose distance is exactly $2r+2s$. If $\gamma$ is a minimizing geodesic between them, consider the balls $B(\gamma(r),r)$ and $B(\gamma(2r+s),s)$. Since \ $d_{CC}(\gamma(u),\gamma(v))=|u-v|$ for any times $u,v$, it follows that these balls are both contained in $B(\gamma(r+s),r+s)$ and are disjoint. Hence \be
\Gamma_M(r)+\Gamma_M(s)=\mu(B(\gamma(r),r))+\mu(B(\gamma(2r+s),s))\leq \mu(B(\gamma(r+s),r+s))=\Gamma_M(r+s), 
\ee as we wanted.

\section{Deduction of the $L^p$ inequality from the $L^1$ inequality}\label{Lp}

We have established Lemma \ref{Faris} and hence Theorem \ref{HPWL1}. We now show how to prove Theorem \ref{HPW} for a general $p<+\infty$.

Fix $f\in C^\infty_c(M)$ and $\alpha>0$ and apply the just proved Theorem \ref{HPWL1} to $f':=|f|^p$ and $\alpha'=p\alpha$ (to be fair, $|f|^p$ is not smooth, but one may replace it with $(|f|^2+\eps)^\frac{p}{2}-\eps^\frac{p}{2}$ in what follows and then pass to the limit as $\eps\rightarrow0$). We obtain
\be
\int_M|f|^pd\mu\leq C\left(p\int_M|\nabla_H f||f|^{p-1}\right)^\frac{p\alpha}{p\alpha+1}\left(\int_Mw^{p\alpha}|f|^p\right)^\frac{1}{p\alpha+1},
\ee where we used Leibniz rule for $\nabla_H$. If we apply H\"older inequality with exponents $p$ and $\frac{p}{p-1}$ to the first integral on the right and then reorder the terms, we find\be
\nor{f}_{L^p(M)}\leq C^\frac{p\alpha+1}{p\alpha+p}p^\frac{\alpha}{\alpha+1}\nor{\nabla_H f}_{L^p(M)}^\frac{\alpha}{\alpha+1}\nor{w^\alpha f}_{L^p(M)}^\frac{1}{\alpha+1},
\ee which is estimate \eqref{hpw} (because $C^\frac{p\alpha+1}{p\alpha+p}$ is bounded uniformly in $p$ and $\alpha$).\newline This concludes the proof of Theorem \ref{HPW} (at least in the non-compact case).

\section{The compact case}\label{compact-section}

Since we described all of our proofs assuming that $M$ is non-compact, it is time to list the modifications one has to make so that everything works also in the compact case. In what follows, the section labelled \ref{compact-section}.$n$ describes how to modify Section $n$ above (here $n\in \{3,4,5\}$). 
Since now $M$ is compact, we may normalize $\mu$ and assume that it is a probability measure. Recall that $d_0$ denotes the diameter of $(M,d_{CC})$.\newline 
\setcounter{subsection}{2}
\subsection{}
The first thing to be modified is Lemma \ref{Faris}, because constant functions are trivial counterexamples to inequality \eqref{faris-estimate} in the compact case. We proceed as follows: if $f\in C^\infty_c(M)$, we define\bel\label{median}
m_0:=\inf\left\{t\in \R:\ \mu\{x\in M:\ f(x)\leq t\}\geq \frac{1}{2}\right\},
\eel which is a \emph{median} for $f$, i.e.\ it satisfies\bel\label{median-ineq}
\mu\{x\in M:\ f(x)> m_0\}\leq \frac{1}{2},\quad \mu\{x\in M:\ f(x)< m_0\}\leq \frac{1}{2}.
\eel 
We now state the substitute to Lemma \ref{Faris}, that will be proved in Subsection \ref{proof-Faris-compact}.
\begin{lem}\label{Faris-compact}
If $E\subseteq M$ is Borel and $\mu(E)\leq \Gamma_M(r)$, then the inequality\be
\int_E|f-m_0|d\mu\leq Cr\int_M|\nabla_Hf|d\mu
\ee holds for every $f\in C^\infty_c(M)$, where $m_0$ is as in \eqref{median}.
\end{lem} 

This implies a \emph{Poincar\'e inequality}.

\begin{cor}\label{poincare}
If $f\in C^\infty_c(M)$ and $\int_Mfd\mu=0$, we have\be
\int_M|f|d\mu\leq C d_0\int_M|\nabla_Hf|d\mu.
\ee 
\end{cor}

\begin{proof}[Proof that Lemma \ref{Faris-compact} implies Corollary \ref{poincare}]
It is enough to notice that, if $f$ has zero average, \be
\int_M|f|d\mu\leq \int_M|f-m_0|d\mu+|m_0-\int_Mfd\mu|\leq 2\int_M|f-m_0|d\mu.
\ee 
\end{proof}

We now show how to prove the $L^1$ inequality of Theorem \ref{HPW} in the compact case. Let $f\in C^\infty_c(M)$ be such that $\int_Mfd\mu=0$. We will prove that \bel\label{compact-r}
\int_M|f|d\mu\leq Cr\int_M|\nabla_Hf|d\mu+Cr^{-\alpha}\int_Mw^\alpha|f|d\mu\qquad\forall r>0.
\eel If $r\geq d_0/8$, this is a trivial consequence of Corollary \ref{poincare}. Assume now that $r\leq d_0/8$:\bee
\int_M|f|d\mu&\leq&\int_{w\leq r}|f-m_0|d\mu+|m_0|\mu(w\leq r)+\int_{w> r}|f|d\mu\\
&\leq & Cr\int_M|\nabla_Hf|d\mu+|m_0|\Gamma_M(r)+r^{-\alpha}\int_Mw^\alpha|f|d\mu,
\eee where we used the hypothesis \eqref{hyphpw} and Lemma \ref{Faris-compact}. 

Observe that the super-additivity of $\Gamma_M$ (Section \ref{tech}) has the following analogue in the compact setting: for every $r,s\geq 0$ \emph{such that $r+s\leq d_0/2$}, $\Gamma_M(r+s)\geq \Gamma_M(r)+\Gamma_M(s)$. In particular this implies $\Gamma_M(d_0/8)\leq 1/4$.
Moreover Markov inequality gives $|m_0|\leq 2\int_M|f|d\mu$. These two facts imply that $|m_0|\Gamma_M(r)\leq \int_M|f|d\mu/2$, which can then be reabsorbed in the left hand side of our estimate, completing the proof of \eqref{compact-r}.

\subsection{}\label{proof-Faris-compact} Our task is to prove Lemma \ref{Faris-compact}. The geometric counterpart of the existence of constant test functions, which causes Lemma \ref{Faris} to fail in the compact case, is the fact that $E=M$ has measure $1$ and perimeter $0$, which causes Theorem \ref{weak-isoperimetric} to fail in the compact case. The required reformulation of Theorem \ref{weak-isoperimetric} is as follows.

\begin{thm}[Weak isoperimetric inequality, compact case]\label{weak-isoperimetric-compact} Let $M$ be a compact manifold satisfying the assumptions of Section \ref{statement-smooth}. If $E\subseteq M$ is a Borel set such that $\min\{\mu(E),\mu(E^c)\}\leq\frac{\Gamma_M(r)}{2}$, then \be
\min\{\mu(E),\mu(E^c)\}\leq Cr\nor{\partial_HE}.
\ee
\end{thm}
See the Appendix for a proof. Theorem \ref{weak-isoperimetric-decoupled} can be reformulated as follows.

\begin{thm}[Weak isoperimetric inequality, stronger form, compact case]\label{weak-isoperimetric-decoupled-compact}\ \newline
If $\mu(E)\leq\Gamma_M(r)$, the following inequality holds for every Borel set $A$ of measure $\leq 1/2$:\be
\mu(A\cap E)\leq Cr\nor{\partial_H A}.
\ee
\end{thm}

The proof is a minor modification of the argument in \ref{tech}. Now we have to adapt to the compact case the argument of Section \ref{isoperimetric-section}, in order to deduce Lemma \ref{Faris-compact} from Theorem \ref{weak-isoperimetric-decoupled-compact}. We write $A_s=\{x\in M: f(x)> s\}$ and take $E\subseteq M$ such that $\mu(E)\leq \Gamma_M(r)$. \bee
\int_E |f-m_0|d\mu&=&\int_{E\cap \{f\geq m_0\}} (f-m_0)d\mu+\int_{E\cap \{f<m_0\}} (m_0-f)d\mu \\
&=&\int_0^{+\infty}\mu(\{f-m_0>s\}\cap E)ds+\int_0^{+\infty}\mu(\{m_0-f\geq s\}\cap E)ds\\
&=&\int_{m_0}^{+\infty}\mu(A_s\cap E)ds+\int_{-\infty}^{m_0}\mu(A_s^c\cap E)ds.
\eee Inequalities \eqref{median-ineq} imply that $\mu(A_s)\leq 1/2$ for $s>m_0$ and $\mu(A_s^c)\leq 1/2$ for $s<m_0$. We can then apply Theorem \ref{weak-isoperimetric-decoupled-compact} and obtain\be
\int_E|f-m_0|d\mu\leq Cr\left( \int_{m_0}^{+\infty}\nor{\partial_HA_s}ds+\int_{-\infty}^{m_0}\nor{\partial_HA_s^c}ds\right).
\ee Notice that $\nor{\partial_HA}=\nor{\partial_HA^c}$ in the compact case, as one can deduce from the definition of the horizontal perimeter. Applying coarea formula, \be
\int_E |f-m_0|d\mu\leq Cr \int_{-\infty}^{+\infty}\nor{\partial_HA_s}ds\leq Cr\int_M|\nabla_Hf|d\mu, 
\ee which is Lemma \ref{Faris-compact}. This concludes the proof of the $L^1$ inequality of Theorem \ref{HPW} in the compact case.

\subsection{} The last step is the deduction of the $L^p$ estimate. Unfortunately the trick of Section \ref{Lp} cannot work in the compact case, since $|f|^p$ does not have average $0$, and we need another twist in our argument. First, we show that the Poincar\'e inequality of Corollary \ref{poincare} has an $L^p$ counterpart, i.e.\ \bel\label{Lppoincare}
\left(\int_M|f|^pd\mu\right)^\frac{1}{p}\leq Cpd_0\left(\int_M|\nabla_Hf|^pd\mu\right)^\frac{1}{p}
\eel holds for any $p<+\infty$ and $f\in C^\infty_c(M)$ with zero average. To see this we notice that if $m_0$ is the median of $f$, then $m_0|m_0|^{p-1}$ is the median of $f|f|^{p-1}$ and \bel\label{poincare-start}
\int_M|f|^pd\mu\leq \int_M\left|f|f|^{p-1}-m_0|m_0|^{p-1}\right|d\mu+\left|\int_M(m_0-f)d\mu\right||m_0|^{p-1},
\eel where we used the fact that $\int_Mfd\mu=0$. Now we apply Lemma \ref{Faris-compact} twice to control the right hand side with\be
Cd_0\int_M\left|\nabla_H(f|f|^{p-1})\right|d\mu+Cd_0\int_M|\nabla_Hf|d\mu\cdot|m_0|^{p-1}.
\ee Since $m_0|m_0|^{p-2}$ is the median of $f|f|^{p-2}$, Markov inequality gives $|m_0|^{p-1}\leq 2\int_M|f|^{p-1}d\mu$ and, by Leibniz and H\"older, we estimate everything by\be
Cpd_0\left(\int_M|\nabla_Hf|^pd\mu\right)^\frac{1}{p}\left(\int_M|f|^pd\mu\right)^{1-\frac{1}{p}}.
\ee Comparing with \eqref{poincare-start}, this proves \eqref{Lppoincare}.\newline 
Analogously to our deduction of \eqref{compact-r}, we have now to prove \bel\label{compact-r-p}
\nor{f}_{L^p(M)}\leq Cpr\nor{\nabla_Hf}_{L^p(M)}+Cr^{-\alpha}\nor{w^\alpha f}_{L^p(M)}
\eel for every $r\leq d_0/8$, since in the complementary interval this is implied by \eqref{Lppoincare}. We do this combining what we did above for the $p=1$ case and for the $L^p$ Poincar\'e inequality:\bee
\nor{f}_{L^p(M)}^p&=&\int_M|f|^pd\mu \\ 
&\leq& \int_{w\leq r}|f|^pd\mu+r^{-\alpha}\int_Mw^\alpha |f||f|^{p-1}d\mu\\
&\leq&\int_{w\leq r}\left|f|f|^{p-1}-m_0|m_0|^{p-1}\right|d\mu+|m_0|^p\Gamma_M(r)\\ 
&&+r^{-\alpha}\left(\int_Mw^{p\alpha}|f|^pd\mu\right)^\frac{1}{p}\left(\int_M|f|^pd\mu\right)^{1-\frac{1}{p}}\\
&\leq&Cpr\int_M|\nabla_Hf||f|^{p-1}d\mu+\frac{2\int_M|f|^pd\mu}{4}\\
&&+r^{-\alpha}\left(\int_Mw^{p\alpha}|f|^pd\mu\right)^\frac{1}{p}\left(\int_M|f|^pd\mu\right)^{1-\frac{1}{p}}\\
&\leq& Cpr\nor{\nabla_Hf}_{L^p(M)}\nor{f}_{L^p(M)}^{p-1}\\
&&+\frac{\nor{f}_{L^p(M)}^p}{2}+r^{-\alpha}\nor{w^\alpha f}_{L^p(M)}\nor{f}_{L^p(M)}^{p-1}.
\eee Reabsorbing the second term and dividing both sides by $\nor{f}_{L^p(M)}^{p-1}$, we find \eqref{compact-r-p}.\newline This concludes the proof of Theorem \ref{HPW} in the compact case too.

\section{Discrete setting}\label{discrete-setting}

%In this Section we study discrete counterparts of the results settled above: the natural setting is that of discrete graphs, with groups acting on them (e.g.\ the Cayley graph of a group). 
In this section we will work on certain discrete homogeneous spaces for finitely generated groups. We begin by describing precisely our setting and then we state our main result.\newline

\par First of all, if $M=(V,E)$ is a countable (finite or infinite) graph, we denote by $\sim$ the adjacency relation and by $|A|$ the counting measure of $A\subseteq V$. We also denote $\ell^p(V)$ the associated Lebesgue spaces of functions on $V$.\newline 
Recall that the \emph{graph distance} is defined by\be
d(m,n):=\inf\text{ length}(\gamma)\qquad (m,n\in V),
\ee where the $\inf$ is taken as $\gamma$ varies over the curves $\gamma:\{0,1,\dots,T\}\rightarrow V$ ($T\in \N$) that connect $m$ to $n$ and are \emph{admissible}, i.e.\ \be
\gamma(i)\sim\gamma(i+1)\qquad\forall i\in\{0,\dots,T-1\}.\ee
 The length of such a curve is given by $T$. The graph metric is a genuine distance if and only if $M$ is connected as a graph. We let $B\!\bra{m,r} = \!\{n\in V: d\!\bra{m,n} <r\}$ be the open ball with center $m\in V$ and radius $r \geq1$ (radii $<1$ are uninteresting in the discrete setting). 

We next define the \emph{gradient} of a function $f: V \to \R$ as follows:
\[\abs{\nabla f} \!\bra{m} := \sum_{n :\ m\sim n } \abs{f\!\bra{m} - f\!\bra{n}} \period\]
Notice that we are not really defining the gradient $\nabla f$, but only its modulus, which is all we need for our results. We will also write $\nor{\nabla f}_{\ell^p(V)}$ for $\nor{\abs{\nabla f}}_{\ell^p(V)}$. In particular, the \emph{edge-perimeter} of a set $A\subseteq M$ is\be
\nor{\partial_E A}:=\nor{\nabla 1_A}_{\ell^1(V)}\in\N\cup\{+\infty\}.
\ee 
%For functions which are not continuous, we may let $\abs{\nabla f } = \infty$.

Assume now that a group $G$ acts on the left on $M$. By this we mean that $G$ acts on the vertex set $V$ and that the action preserves the graph structure, i.e.\ $m\sim n$ if and only if $x\cdot m\sim x\cdot n$ for any $x\in G$ and $m,n\in V$. We adopt the same notation as in the smooth setting, writing $\phi_x(m):=x\cdot m$ for the action of $x\in G$ on $m\in M$. In such a case, the distance $d$ is $G$-invariant, i.e.\  
\[ d\!\bra{x\cdot m, x\cdot n} = d\!\bra{m,n}  \quad \forall x \in G,\ \forall m,n \in V\]
and the modulus of the gradient is $G$-invariant, i.e.\ \be
|\nabla(f\circ\phi_x)|=|\nabla f|\circ\phi_x\qquad\forall x\in G,\ \forall f:V\rightarrow \C.
\ee 

%More facts on lengths, distances and gradients in the $G$-invariant setting will be investigated below. We now address the main result.

\subsection{Standing assumptions and statement of the main result in the discrete setting}\label{assumptions-discrete}

We state our \emph{standing assumptions} in the discrete setting:
\begin{enumerate}[(i)]
\item $M=(V,E)$ is a connected countable (finite or infinite) graph, 
\item $G$ is a finitely generated group,
\item $G$ has a transitive left action on $M$ (and in particular on $V$),
\item the isotropy subgroup of some, hence every, point of $V$ is finite,
\item the degree of some, hence every, vertex of $M$ is finite.
\end{enumerate}

Recall that the degree of a vertex of a graph is the number of vertices adjacent to it. We let $\delta$ be the degree of the vertices of $M$. According to the previous section, a graph $M$ satisfying the standing assumptions has a $G$-invariant distance $d$ and there is a $G$-invariant gradient $\nabla$ for functions on $V$. 

The main example to keep in mind is given by the Schreier coset graph associated to a finitely generated group $G$, a finite subgroup $K$ and a finite symmetric generating set $S\subseteq G$. This means that $V:=G\slash K$, on which $G$ acts on the left, $xK\sim yK$ if and only if $xK\neq yK$ and $yK=sxK$ for some $s\in S$. If $K$ is the trivial subgroup, this is nothing but the Cayley graph of $(G,S)$. Notice that these graphs are connected because $S$ is generating, and that $\delta \le |S|$. 

Observe that the counting measure provides a $G$-invariant measure on $V$, which is easily seen to be unique up to positive multiples (the implicit $\sigma$-algebra is the discrete one). We introduce therefore the \emph{volume growth function} of open balls,
\[ \Gamma_M\!\bra{r} := \abs{  B\!\bra{m,r} } \qquad\forall r \ge 1, \] which is well-defined by $G$-invariance and transitivity.

We state the main result in the discrete setting. 

\begin{thm}\label{HPW-discrete} Assume that $M$ satisfies the assumptions above and that $w:V\rightarrow[1,+\infty)$ is a function such that\be
\abs{\cur{m \in V:\ w(m)\leq r}} \leq \Gamma_M(r)\qquad\forall r\geq1.
\ee
If $V$ is infinite, then the inequality
\begin{equation}
\label{hpw-discrete}
\nor{f}_{\ell^p\!\bra{V}} \le C p^{\frac{\alpha}{\alpha+1} }\nor{ \nabla f }_{\ell^p\!\bra{V}}^{\frac{\alpha}{\alpha+1} } \nor{w^\alpha f}^{\frac{1}{\alpha +1}}_{\ell^p\!\bra{V}}
\end{equation}
holds for every finitely supported $f:V\rightarrow\R$ and $\alpha>0$. If $V$ is finite, the inequality holds for every $f$ that satisfies the additional assumption
\[ \sum_{m\in V}f(m) = 0\period \]
\end{thm}
We make a few comments specific to the discrete setting.
\begin{enumerate}[(i)]
%\item This result improves the trivial  estimate  $\nor{f}_{\ell^p(V)} \le \nor{w f}_{\ell^p(V)}$ introducing the energy term $\nor{ \nabla f }_{\ell^p\!\bra{V}}$.

\item The main qualitative information that we obtain from this result is that, if $\nor{f}_{\ell^p(V)} = 1$ and the $\ell^p$-energy $\nor{ \nabla f }_{\ell^p(V)}$ is small, then $f$ cannot be very concentrated. In fact H\"older inequality gives $\nor{\nabla f}_{\ell^p(V)}\leq 2\delta\nor{f}_{\ell^p(V)}$, and hence the $\ell^p$-energy cannot be large,  if $f$ is normalized in $\ell^p(V)$. 

\item Some assumption $w \ge w_0 >0$ is necessary, as one can see by testing \eqref{hpw-discrete} with a Dirac delta.

\item An application of Cauchy-Schwarz inequality gives
\[ \nor{\nabla f}_{\ell^2(V)}^2 \le 2\delta \sum_{m\sim n}|f(m)-f(n)|^2,\]
so, when $p=2$, we may deduce from this an inequality with the Dirichlet energy in place of the gradient.

\end{enumerate}

\section{Proof of Theorem \ref{HPW-discrete}}\label{proof-discrete} 
The global structure of the proof of Theorem \ref{HPW-discrete} is the same as the one of the proof of Theorem \ref{HPW}, but the local details need a routine translation, for which the dictionary below may help:\newline

\begin{center}
\begin{tabular}{|c| c|}
\hline
Smooth & Discrete\\
\hline
$C^\infty_c(M)$ & finitely supported on $V$\\
$\mu$ & counting measure $|\cdot|$  \\
%$\nu$ & $(\text{card. of isotropy subgroups})^{-1}|\cdot|$   \\
$A\subseteq M$ Borel & $A\subseteq V$\\
$\int_A$ & $\sum_{m\in A}$\\
$L^p(M)$ & $\ell^p(V)$\\
$\nabla_Hf$ & $|\nabla f|$\\
$d_{CC}$ & $d$\\ 
radii $r\geq0$ & radii $r\geq 1$\\
$\nor{\partial_HA}$ & $\nor{\partial_EA}$\\
$w:M\rightarrow[0,+\infty]$ Borel & $w:V\rightarrow[1,+\infty]$ \\
%$\mu(E)=\Gamma_M(r)$ & $\abs{E}\leq \Gamma_M(r)$\\[0.5ex]
\hline
\end{tabular}
\end{center}
For example, Theorem \ref{HPW-discrete} is nothing but the translation of Theorem \ref{HPW} according to our dictionary. For the sake of clarity, we explicit a few observations.\begin{enumerate}[(i)]
\item Coarea formula \eqref{co} in Section \ref{isoperimetric-section} is formally the same, but its proof is much easier and follows from the identity\be
\int_0^{+\infty}|1_{\{|f|>s\}}(m)-1_{\{|f|>s\}}(n)|ds=\left||f(m)|-|f(n)|\right| \qquad\forall m,n\in V,
\ee summing over adjacent $m,n\in V$.
\item The discrete weak isoperimetric theorems, i.e. the translation of Theorem \ref{weak-isoperimetric} and Theorem \ref{weak-isoperimetric-compact}, may be proved translating the content of the Appendix. 
%\item The proof of Theorem \ref{weak-isoperimetric-decoupled} in \ref{tech} is modified as follows. If $\abs{E} \le \Gamma_M\!\bra{r}$, we might assume that $A \cap E$ is not empty and so it holds $\abs{A \cap E } \ge \Gamma_M\!\bra{1}$. Let therefore $m\ge 1$ be the minimal integer such that $\abs{A \cap E} \le \Gamma_M(m)$ and let $n\ge 1$ be the minimal integer such that $\abs{A} \le \Gamma_M(n)$. Notice that $m \le n$ and $m\le r$. The discrete counterpart of Theorem \ref{weak-isoperimetric} yields\be
%\abs{A}\leq C n \nor{\partial_E A}.
%\ee If $m=n$, we conclude. Otherwise, $n\ge 2$ and, denoting $k$ the integral part of $(n-1)/m$, super-additivity of $\Gamma_M$ (i.e. the appropriate translation) gives\be
%\frac{\abs{A}}{n} > \frac{\Gamma_M(n-1)}{n} \geq \frac{\Gamma_M(km)}{4km}\geq \frac{k\Gamma_M(m)}{4km}\geq \frac{\abs{A\cap E}}{4r}.
%\ee 

\item The argument in Section \ref{Lp} can be adapted in a straightforward way using a discrete Leibniz rule, whose proof we sketch here for the reader's convenience. In the inequality
\[ \abs{ x\abs{x}^{p-1} - y\abs{y}^{p-1} } \le p \abs{ x - y } \bra{ \abs{x}^{p-1} + \abs{y}^{p-1}}, \]
which holds for every $x$, $y \in \cnum$, we let $x=f(m)$, $y= f(n)$ and sum on $m$, $n \in V$ with $m\sim n$. We obtain
\[ || \nabla ( f \abs{f}^{p-1}) ||_{\ell^1\bra{V}} \le p \sum_{m\sim n} \abs{ f(m) - f(n) }\! \bra{ \abs{f(m)}^{p-1} + \abs{f(n)}^{p-1}} \comma \]
which entails the estimate
\[ \nor{ \nabla |f|^p }_{\ell^1\bra{V}}  \le|| \nabla ( f \abs{f}^{p-1}) ||_{\ell^1\bra{V}}\le  2p \nor{ |f|^{p-1} \abs{\nabla f}} _{\ell^1\bra{V}} \period\]
\end{enumerate}

\section{Constants in uncertainty and isoperimetric inequalities}\label{question}

In this section we briefly comment on the quantitative aspect of the connection between isoperimetric and uncertainty inequalities. Given a smooth or discrete homogeneous space $M$ satisfying our standing assumptions, we define $C_M$ as the smallest positive number such that 
\be
\mu(E)\leq C_Mr\nor{\partial E}\qquad\forall E\subseteq M: \mu(E)\leq \Gamma_M(r),\quad\forall r,
\ee where $\nor{\partial E}$ is the appropriate perimeter and $\mu$ the invariant measure. Theorem \ref{weak-isoperimetric} implies that $C_M$ is bounded independently of $M$. As already remarked, $C_{\R^N}=N^{-1}$. 

We next define $D_M$ as the smallest positive number such that \be \nor{f}_{L^2(M)}^2\leq D_M\nor{\nabla f}_{L^2(M)} \nor{d\cdot f}_{L^2(M)},\ee where $d$ is the appropriate $G$-invariant distance. For the sake of simplicity we do not consider more general $w$'s, or different values of $p$, but one can easily extend what follows to cover these more general cases. By \eqref{heisen}, we know that $D_{\R^N}=4N^{-2}$. Keeping track of the constants in our arguments, one can easily see that in general $D_M\leq C \cdot C_M^2$. Bounding $C_M$ could be an easier task than actually solving the isoperimetric problem on $M$, i.e.\ describing the sets with minimal perimeter among those of a given measure.

\appendix

\section{Proofs of Theorem \ref{weak-isoperimetric} and Theorem \ref{weak-isoperimetric-compact}}\label{isoperimetric-proof-section}
In this appendix we assume that $M$ satisfies the assumptions of Section \ref{statement-smooth}. In particular, $M$ is not assumed to be compact, unless otherwise specified.
Before starting, we fix a Haar measure $\nu$ on $G$ and observe that if $m\in M$ and \bee
\pi_m: G&\longrightarrow& M\\
x&\longmapsto&\phi_x(m):=x\cdot m,\eee then $\mu_m(E):=\nu(\pi_m^\leftarrow(E))$ ($E\subseteq M$ Borel) defines a $G$-invariant measure which is finite on compact sets, since isotropy subgroups are assumed to be compact. Unimodularity of $G$ easily implies that $\mu_m$ does not depend on $m$ and we accordingly suppress the subscript $m$ from the notation. Notice that the $G$-invariant measure on $M$ is unique up to multiplicative constants (see Theorem $2.49$ of \cite{fo-abstract}), so there is no loss of generality in working with this measure $\mu$ .

\subsection{}\label{double-counting} \emph{If $A,B\subseteq G$ are Borel subsets such that $\nu(A)\leq \nu(B)/2<+\infty$, then} \be\frac{1}{\nu(B)}\int_B\nu(Ab\setminus A)d\nu(b)\geq  \frac{\nu(A)}{2}.\ee

Notice that this statement has a probabilistic interpretation: if $B$ is significantly larger than $A$, the right-translates of $A$ by a random element of $B$ are on average significantly disjoint from $A$. Its proof is essentially combinatorial, being a continuous double-counting argument. We proceed with the details.

Consider the following subset of $G\times G$:\be
C:=\{(a,b)\in A\times B: ab\notin A\}.
\ee If $m:G\times G\rightarrow G$ denotes group multiplication, then $C=m^\leftarrow(G\setminus A)\cap (A\times B)$, and hence $C$ is a Borel subset of $G\times G$ with respect to the product topology. Recall that the Borel $\sigma$-algebra of $G\times G$ equals the product of the Borel $\sigma$-algebras of the two factors, since $G$ is second countable. If we endow $G\times G$ with the product measure $\nu\times \nu$, we can then apply Fubini Theorem, obtaining\be
\int_A \nu(b\in B: ab\notin A)d\nu(a)=\int_B \nu(a\in A: ab\notin A)d\nu(b),
\ee which can be rewritten as follows\be
\int_A \nu(B\setminus a^{-1}A)d\nu(a)=\int_B \nu(A\setminus Ab^{-1})d\nu(b).
\ee Now the assumptions on the measures $A$ and $B$ and left invariance give \be\nu(B\setminus a^{-1}A)\geq \nu(B)-\nu(A)\geq \frac{\nu(B)}{2}.\ee This, combined with the equality above, gives\be
\frac{1}{\nu(B)}\int_B \nu(Ab\setminus A)d\nu(b)=\frac{1}{\nu(B)}\int_B \nu(A\setminus Ab^{-1})d\nu(b)\geq \frac{\nu(A)}{2},
\ee where we used the right invariance of $\nu$ in the first equality.
\newline
\par We now fix $0\in M$ and if $f$ is a function on $M$ and $E\subseteq M$, we let $\widetilde{f}:=f\circ\pi_0$ and $\widetilde{E}:=\pi_0^\leftarrow(E)$. From now on, every set is assumed to be Borel.

\subsection{}\label{grad-dist} \emph{If $f\in C^\infty_c(M)$ and $x\in G$, it holds}
\bel\label{lifting} \int_G \abs{ \widetilde{f}(yx) - \widetilde{f}(y) }d\nu(y) \le d_{CC}(0,x\cdot 0)\int_M|\nabla_Hf(m)|d\mu(m).\eel

Let $\gamma: [0,T] \to M$ be a piecewise $C^1$ horizontal curve of length $L$ connecting $0$ to $x\cdot 0$. Then it holds for every $y\in G$, by the Fundamental Theorem of Calculus, the $G$-invariance of $\nabla_H$ and Cauchy-Schwarz,
\bee |\widetilde{f} (yx) - \widetilde{f}(y)|  &=& |f\circ\phi_y(\gamma(T))-f\circ\phi_y(\gamma(0))|\\
\abs{\int _0^T  g(\nabla_H(f\circ\phi_y)(\gamma(t)) , \dot{\gamma}(t)) dt} 
&\le& \int _0^T |\nabla_H f |(y\cdot \gamma(t))\cdot|\dot{\gamma}(t)| dt\period \eee
Integrating in $y$ with respect to $\nu$ and using Fubini, we obtain
\bee \int_G |\widetilde{f}(yx)- \widetilde{f}(y)|d\nu(y) &\leq& \int_0^T|\dot{\gamma}(t)| \left(\int_G |\nabla_H  f(\pi_{\gamma(t)}(y))| d\nu(y)\right)dt\\
&=&L\int_M |\nabla_H  f(m)| d\mu(m),
\eee where the last identity follows from our observation that $\mu$ is the push-forward of $\nu$ with respect to $\pi_m$ for any $m$.
Taking the $\inf$ over all the horizontal curves connecting $0$ to $x\cdot 0$, we find the thesis. Estimate \eqref{lifting} also appears in Section $6$ of \cite{pseudo}, where the author deals with pseudo-Poincar\'e and Sobolev inequalities.

\subsection{}\label{grad-dist-set}\emph{For any Borel set $E\subseteq M$ of finite $\mu$-measure and $x\in G$,}\be
\nu(\widetilde{E}x\triangle \widetilde{E})\leq d_{CC}(0,x\cdot 0)\nor{\partial_H E}.
\ee 

Here $\triangle$ denotes the symmetric difference of sets. Notice that this is the version for sets of the previous step.

 Take any sequence $f_n\in C^\infty_c(M)$ converging in $L^1(M)$ to $1_E$. Now \ref{grad-dist} gives the inequalities\be
\int_G \abs{ \widetilde{f}_n(y\cdot x^{-1}) - \widetilde{f}_n(y) }d\nu(y) \le d_{CC}(0,x\cdot 0)\int_M|\nabla_Hf_n(m)|d\mu(m),
\ee where we used $d_{CC}(0,x\cdot 0)=d_{CC}(0,x^{-1}\cdot 0)$, a consequence of $G$-invariance of the distance. Taking the $\liminf$ of both sides, we find\be
\nu(\widetilde{E}x\triangle\widetilde{E})=\int_G \abs{ 1_{\widetilde{E} x}- 1_{\widetilde{E}}}d\nu\le d_{CC}(0,x\cdot 0)\liminf_{n\rightarrow+\infty}\int_M|\nabla_Hf_n(m)|d\mu(m).
\ee Taking the $\inf$ with respect to the choice of the sequence $f_n$, we find what we wanted.

\subsection{}\label{almost}\emph{If $\mu(E)\leq\Gamma_M(r)/2$, then}\be
\mu(E)\leq r\nor{\partial_HE}.
\ee
Let $x\in B:=\pi_0^\leftarrow(B(0,r))$: by right-invariance of $\nu$ and \ref{grad-dist-set} we have\be
\nu(\widetilde{E}x\setminus \widetilde{E})+\nu(\widetilde{E}x^{-1}\setminus \widetilde{E})=\nu(\widetilde{E}x\setminus \widetilde{E})+\nu(\widetilde{E}\setminus \widetilde{E}x)=\nu(\widetilde{E}x\triangle \widetilde{E})\leq r\nor{\partial_H E}.
\ee Since $\nu(\widetilde{E})=\mu(E)\leq \mu(B(0,r))/2= \nu(B)/2$, we can average the left hand side with respect to $x\in B$ and conclude by an application of \ref{double-counting}. 

What we said until now gives a proof of Theorem \ref{weak-isoperimetric-compact}.

\subsection{}\label{weaker}\emph{Theorem \ref{weak-isoperimetric} holds, i.e.\ if $M$ is non-compact, $E\subseteq M$ and $\mu(E)\leq\Gamma_M(r)$, then}
\be
\mu(E)\leq 2r\nor{\partial_H E}.
\ee
By super-additivity of $\Gamma_M$ (see Section \ref{tech}), $\Gamma_M(r)\leq \Gamma_M(2r)/2$. We can then apply \ref{almost} and conclude.

\section*{Acknowledgements}

The authors would like to express their gratitude to F. Ricci, who introduced them to the subject of uncertainty inequalities and read very carefully a draft of the present paper.

We also thank L. Ambrosio, for his interesting comments on the isoperimetric side of these matters, and S. Di Marino, who pointed to us the work of M. Miranda.

\bibliographystyle{amsalpha}
\bibliography{biblio-uncertainty}

\end{document}